\documentclass[reqno]{amsart}

\usepackage[T1, T2A]{fontenc} 
\usepackage[utf8]{inputenc}

\usepackage[foot]{amsaddr}

\usepackage{amsthm}
\usepackage{amssymb}
\usepackage{mathtools}
\usepackage{enumitem}
\usepackage{xcolor}
\usepackage{xparse}
\usepackage{soul}

\usepackage[unicode, pdftex]{hyperref}
\hypersetup{
    colorlinks=true,
    urlcolor=magenta,
    citecolor=blue,
    linkcolor=red
}

\newtheorem{theorem}{Theorem}

\newtheorem{lemma}{Lemma}

\newtheorem{remark}{Remark}

\newcommand{\bfeta}{{\boldsymbol{\eta}}}

\NewDocumentCommand{\PR}{o o} {
    \IfNoValueTF{#2} {
        \IfNoValueTF{#1} {
            \mathbf{P}
        } {
            \mathbf{P} \left( #1 \right) 
        }
    } {
        \mathbf{P} \left( \left. #1 \right| #2 \right)
    }
}

\NewDocumentCommand{\PRB}{o o} {
    \IfNoValueTF{#2} {
        \IfNoValueTF{#1} {
            \mathbf{P}_{\bfeta}
        } {
            \mathbf{P}_{\bfeta} \left( #1 \right) 
        }
    } {
        \mathbf{P}_{\bfeta} \left( \left. #1 \right| #2 \right)
    }
}

\NewDocumentCommand{\ER}{o o} {
    \IfNoValueTF{#2} {
        \IfNoValueTF{#1} {
            \mathbf{E}
        } {
            \mathbf{E} \left( #1 \right)
        }
    } {
        \mathbf{E} \left( \left. #1 \right| #2 \right)
    }
}

\NewDocumentCommand{\ERB}{o o} {
    \IfNoValueTF{#2} {
        \IfNoValueTF{#1} {
            \mathbf{E}_{\bfeta}
        } {
            \mathbf{E}_{\bfeta} \left( #1 \right)
        }
    } {
        \mathbf{E}_{\bfeta} \left( \left. #1 \right| #2 \right)
    }
}

\NewDocumentCommand{\DR}{o o} {
    \IfNoValueTF{#2} {
        \IfNoValueTF{#1} {
            \mathrm{Var}
        } {
            \mathrm{Var} \left( #1 \right)
        }
    } {
        \mathrm{Var} \left( \left. #1 \right| #2 \right)
    }
}

\NewDocumentCommand{\DRB}{o o} {
    \IfNoValueTF{#2} {
        \IfNoValueTF{#1} {
            \mathrm{Var}_{\bfeta}
        } {
            \mathrm{Var}_{\bfeta} \left( #1 \right)
        }
    } {
        \mathrm{Var}_{\bfeta} \left( \left. #1 \right| #2 \right)
    }
}

\DeclarePairedDelimiter\abs{\lvert}{\rvert}

\title[
    On Asymptotic Behavior of Extinction Moment
    of Critical BBPRE
]{
    On Asymptotic Behavior of Extinction Moment
    of Critical Bisexual Branching Process
    in Random Environment
}

\author{Zhiyanov A.P.}
\author{Shklyaev A.V.}
\address[Z, S]{Faculty of Mechanics and Mathematics, Lomonosov Moscow State University, Moscow, 119991, Russia}
\email[Z, S]{zhiyanovap@gmail.com, alexander.shklyaev@math.msu.ru}

\date{12.11.24}

\begin{document}
\maketitle
\begin{abstract}
We consider a critical bisexual branching process in a random environment generated by independent and identically distributed random variables.
Assuming that the process starts with a large number of pairs $N$, we prove that its extinction time is of the order $\ln^2 N$.
Interestingly, this result is valid for a general class of mating functions.
Among them are the functions describing the monogamous and polygamous behavior of couples, as well as the function reducing the bisexual branching process to the simple one.

\end{abstract}

\section{Introduction}
Let $\bfeta = \{\eta_1, \eta_2, \ldots\}$ be a sequence of independent and identically distributed (i.i.d.) random variables (r.v.), where the term $\eta_n \in \mathbb{R}$ denotes the random environment at moment $n$.
Also, consider the family $\{ f_{\alpha}(s, t),\ \alpha \in \mathbb{R} \}$ of probability generating functions, and introduce the mating function $L: \mathbb{N}_{0} \times \mathbb{N}_{0} \times \mathbb{R} \rightarrow \mathbb{N}_{0}$, where $\mathbb{N}_0 = \mathbb{N} \cup \{ 0 \}$.

Under this notation, we define a bisexual branching process in a random environment (BBPRE) $\{ N_n,\ n \geq 0 \}$ as a time-homogeneous Markov chain with transition probabilities depending on $\bfeta$ by the rule:
\[
    \PR[ N_{n + 1} = j\, ][ N_n = i, \bfeta ] = \PR[ L\left( F_{n + 1}, M_{n + 1}, \eta_{n + 1} \right) = j\, ][ \eta_{n + 1} ],
\]
where $N_0 = N$, and the distribution of the random pair $(F_{n + 1}, M_{n + 1})$ conditioned on $\eta_{n + 1}$ has the probability generating function $f_{\eta_{n + 1}}(s, t)^i,\ s, t \in [0, 1]$.

More intuitively, the process $\{ N_n,\ n > 0 \}$ can be defined as follows.
Let us consider a sequence $\{ (F_{i, j}, M_{i, j}),\ i, j > 0 \}$ of two-dimensional random vectors taking values in $\mathbb{N}_0 \times \mathbb{N}_0$.
For a fixed environment at moment $n$, we assume the pairs $\{ (F_{n, j}, M_{n, j}),\ j > 0 \}$ to be conditionally independent and identically distributed with the corresponding probability generating function $f_{\eta_n}(s, t)$.
Also, the vectors $\{ (F_{n, j}, M_{n, j}),\ j > 0 \}$ are assumed to be independent for different $n \in \mathbb{N}$.
Under this notation, we define a BBPRE as the sequence
\begin{equation}
    N_0 = N,\ N_{n + 1} = L\left( \sum_{j = 1}^{N_n} F_{n + 1, j}, \sum_{j = 1}^{N_n} M_{n + 1, j}, \eta_{n + 1} \right),\ n \geq 0.
    \label{eq:bbpre-definition}
\end{equation}

This definition allows us to consider this process step by step using a biological analogy.
For a fixed moment $n$, $N_n$ couples start to produce offspring.
The $j$th couple of $n$th generation produces $F_{n + 1, j}$ ``female'' and $M_{n + 1, j}$ ``male'' descendants, respectively.
Then, all new ``female'' and ``male'' descendants group in 
\[
    N_{n + 1} = L\left(  \sum_{j = 1}^{N_n} F_{n + 1, j}, \sum_{j = 1}^{N_n} M_{n + 1, j}, \eta_{n + 1} \right)
\]
couples, which in turn, produce offspring.
Under this interpretation, the standard mating functions
\begin{equation}
    L(x, y, z) = x \min(1, y);\ \  
    L(x, y, z) = \min(x, y\, d(z)),\ d: \mathbb{R} \rightarrow \mathbb{N}.
    \label{eq:mating-standard}
\end{equation}
describe polygamous and monogamous ($d(z) = 1$) mating schemes, respectively.

Bisexual branching processes are considered to be a natural modification of a standard Galton-Watson branching process.
They were first introduced in~\cite{Daley1968} for the standard mating functions~\eqref{eq:mating-standard}.

In that study, the sufficient conditions for process extinction with probability one were provided.
Further advances in the extinction problem were made in~\cite{Hull1982, Bruss1984} for \emph{superadditive} mating functions.

Bisexual branching processes in a random environment were introduced in~\cite{Ma2006},
where non-degeneracy conditions of such processes were provided.
In~\cite{Ma2009}, the previous results were generalized to the case where the mating function depends on the environment.
Finally, in~\cite{Xiao2023}, an almost sure (a.s.) convergence rate of a normalized \emph{supercritical} BBPRE was studied under the condition
\[
    \prod_{n = 1}^{\infty} \ER[ \frac{ \sup_{ k > 0} k^{-1} \ERB[ N_{n + 1} ][ N_n = k ] }{ \inf_{ k > 0 } \, k^{-1} \ERB[ N_{n + 1} ][ N_n = k ] } ] < \infty.
\]
Unfortunately, it is too strict for the case of i.i.d. environments, since it implies that
$ 
    \ERB[ N_1 ][ N_0 = k] 
$
does not depend on $k$.
We consider much weaker condition that the mating function can be approximated by a Lipschitz function.

Note that BBPRE can be considered a natural generalization of controlled branching processes.
Indeed, let 
$
    L(x, y, z) = h(x, z),
$
where $h: \mathbb{N}_0 \times \mathbb{R} \rightarrow \mathbb{N}_0$ is a measurable function.
Then, the corresponding BBPRE reduces to a controlled process.
For a comprehensive review of controlled branching process results, see~\cite{Velasco2017}.

Denote by $\mathbb{R}_{\geq 0} = \{x \in \mathbb{R}\ |\ x \geq 0 \}$, and let $\{ N_n,\ n \in \mathbb{N}_0 \}$ be a BBPRE starting with $N_0 = N$ pairs.
Also, we will use the notation $\PRB[ \cdot ]$ and $\ERB[ \cdot ]$ instead of the conditional probability $\PR[ \cdot ][ \bfeta ]$ and expectation $\ER[ \cdot ][ \bfeta ]$, respectively.
We require the mating function $L(x, y, z)$ to satisfy the following conditions:
\begin{enumerate}[label=C\arabic*), ref=C\arabic*]
    \item \label{cond:mating-superadditive}
    $L(x, y, z)$ is \emph{superadditive}, i.e., for any $x, y, u, v \in \mathbb{N}_0$ and $z \in \mathbb{R}$,
    \[
        L(x + u, y + v, z) \geq L(x, y, z) + L(u, v, z);
    \]
    \item \label{cond:mating-apprx-lipsh}
    there is a function 
    $
        g(x, y, z): \mathbb{R}_{\geq 0} \times \mathbb{R}_{\geq 0} \times \mathbb{R} \rightarrow \mathbb{R}_{\geq 0},
    $
    that is Lipschitz with respect to $(x, y)$, i.e., there is a measurable function 
    $
        \lambda(z): \mathbb{R} \rightarrow \mathbb{R}_{> 0}
    $
    such that for any $x, y, u, v \in \mathbb{R}_{\geq 0}$ and $z \in \mathbb{R}$,
    \[
        \abs*{g(x, y, z) - g(u, v, z)} \leq \lambda(z) \left( \abs*{x - u} + \abs*{y - v} \right);
    \]
    \item \label{cond:mating-apprx-scale}
    $g(t x, t y, z) = t g(x, y, z)$ for all $t \geq 0$, $x, y \in \mathbb{R}_{\geq 0}$, and $z \in \mathbb{R}$;
    \item \label{cond:mating-apprx-apprx}
    there are a measurable function 
    $
        \rho(z): \mathbb{R} \rightarrow \mathbb{R}_{> 0}
    $
    and a constant $0 < \alpha < 1$ such that 
    \[
        \abs*{ L\left( x, y, z \right) - g\left( x, y, z \right) } \leq \rho(z) \left( x + y \right)^{\alpha},\ x, y \in \mathbb{N}_0,\ z \in \mathbb{R}.
    \]
\end{enumerate}
Similar requirements on $L(x, y, z)$ were introduced in~\cite{Shklyaev2023}, where large deviations of BBPRE were studied.

Let $\varphi_n = F_{n, 1}$ and $\mu_n = M_{n, 1}$ be the offspring sizes of the first pair.
Denote by $\ln^+ x = \max(0, \ln x)$, $x > 0$.
Now, we introduce several technical conditions.
\begin{enumerate}[resume*]
    \item \label{cond:bbpre-technical-step}
    There is a constant $\beta > 1$ such that 
    $
        \ERB \varphi^{\beta}_1, \ERB \mu^{\beta}_1
    $
    are finite a.s.
\end{enumerate}

\begin{remark}
    Without loss of generality, we assume $\alpha$ from~\ref{cond:mating-apprx-apprx} to satisfy the inequality $\alpha^{-1} < \beta$.
    So, we will use the notation $\delta = \alpha^{-1} - 1$, where $1 < 1 + \delta < \beta$.
\end{remark}

\begin{enumerate}[resume*]
    \item \label{cond:bbpre-technical-resid}
    For all $n \geq 1$
    $
        \ER \abs*{ \xi_n }^{1 + \beta} < \infty
    $
    and
    $
        \ER \abs*{ \zeta_n }^{1 + \beta} < \infty,
    $
    where 
    \begin{align*}
        \xi_n &= \ln g \big( \ERB \varphi_n, \ERB \mu_n, \eta_n \big), \
        \zeta_n = \ln^{+} \big( \omega^1_n + \omega^2_n + \omega^3_n \big), \\
        \omega^1_n &= \lambda^{1 + \delta}(\eta_n) + \rho^{1 + \delta}(\eta_n), \
        \omega^2_n = \ERB \varphi_n + \ERB \mu_n, \\
        \omega^3_n &= \ERB \abs*{ \varphi_n - \ERB \varphi_n }^{1 + \delta} + \ERB \abs*{ \mu_n - \ERB \mu_n}^{1 + \delta}.
    \end{align*}
\end{enumerate}

\begin{remark}
    Note that, due to Minkowski inequality and non-negativity of the terms included in the definition of $\omega_n^1$, $\omega_n^2$ and $\omega_n^3$,
    $
        \ER \abs*{ \zeta_n }^{1 + \beta}
    $
    is finite if the moments
    \begin{align*}
        \ER \left( \ln^{+} \lambda(\eta_n) \right)^{1 + \beta}, \
        \ER \left( \ln^{+} \rho(\eta_n) \right)^{1 + \beta}, \
        &\ER \left( \ln^{+} \ERB \varphi_n \right)^{1 + \beta}, \
        \ER \left( \ln^{+} \ERB \mu_n \right)^{1 + \beta}, \\
        \ER \left( \ln^{+} \ERB \abs*{ \varphi_n - \ERB \varphi_n }^{1 + \delta} \right)^{1 + \beta}, \
        &\ER \left( \ln^{+} \ERB \abs*{ \mu_n - \ERB \mu_n }^{1 + \delta} \right)^{1 + \beta}
    \end{align*}
    are finite.
\end{remark}

Under these conditions, we introduce the notation $\sigma^2 = \DR[ \xi_n ]$ and consider the random walk $S_n = \sum_{i = 1}^{n} \xi_n$, $S_0 = 0$, associated with the BBPRE.
We require its increments to satisfy the condition
\begin{enumerate}[resume*]
    \item \label{cond:bbpre-criticality}
    the process $N_n$ is \emph{critical}, i.e., $\ER \xi_n = 0$.
\end{enumerate}
Now, we are ready to introduce the extinction moment $\tau$ of a BBPRE
\[
    \tau = \min\{ n \in \mathbb{N} \ \left| \ N_n = 0 \right. \},\ N_0 = N.
\]

An asymptotic behavior of $\tau$ for a bisexual branching process with large initial number of pairs but without environment was provided in~\cite{Abdiushev2024} for the \emph{subcritical} case.
The \emph{critical} case of a Galton-Watson branching process was studied in~\cite{Lamperti1967}, while the case with the environment was recently considered in~\cite{Afanasev2024}.

In this work we study asymptotic behaviour of $\tau$ for a \emph{critical} BBPRE starting with large initial number of pairs $N_0 = N$ (Theorem~\ref{theorem:main}).
In Section~\ref{sec:auxiliary} we provide auxiliary lemmas used in the proof of Theorem~\ref{theorem:main}, while Section~\ref{sec:main} contains the proof itself.

\section{Auxiliary results}
\label{sec:auxiliary}

\begin{lemma}[Exercise 1.6.3 of~\cite{Durrett2019}]
    Let $\xi$ be a r.v. with $\ER \xi^{\alpha} < \infty$, where $\alpha \geq 1$.
    Then,
    \[
        n^{\alpha}\, \PR[ \xi > n ] \rightarrow 0,
    \]
    as  $n \rightarrow \infty$.
    \label{lemma:proba-asymptotics}
\end{lemma}

\begin{lemma}[Minkowski inequality]
    For any $x_1, \ldots, x_n \geq 0$ and $\alpha \geq 1$, the inequality
    \[
        \left(\sum_{i = 1}^n x_i \right)^{\alpha} \leq n^{\alpha - 1} \sum_{i = 1}^n x_i^{\alpha}
    \]
    holds.
    \label{lemma:jen-ineq}
\end{lemma}

\begin{lemma}[Marcinkiewicz-Zygmund inequality,~\cite{Marcinkiewicz1938}]
    Let $\{ \xi_n \}_{n = 1}^{\infty}$ be a sequence of i.i.d. r.v. with $\ER \abs*{\xi_n}^{\alpha} < \infty$, where $1 \leq \alpha \leq 2$, and $\ER \xi_n = 0$ for all $n \in \mathbb{N}$.
    Then, there is a constant $c_{\alpha} > 0$ such that
    \[
        \left( \ER \abs*{ \sum_{i = 1}^n \xi_i } \right)^{\alpha}
        \leq \ER \abs*{ \sum_{i = 1}^n \xi_i }^{\alpha}
        \leq c_{\alpha} \sum_{i = 1}^n \ER \abs*{\xi_i}^{\alpha}.
    \]
    \label{lemma:zygmund-marcinkiewicz}
\end{lemma}

\section{Main results}
\label{sec:main}

\begin{theorem}
    Let $N_n$ be a critical BBPRE with $N_0 = N$ pairs satisfying Conditions~\ref{cond:mating-superadditive}~--~\ref{cond:bbpre-criticality}.
    Then, the extinction moment 
    $
        \tau = \min\{ n \in \mathbb{N} \ \left| \ N_n = 0 \right. \}
    $
    satisfies the relation
    \[
        \frac{ \tau }{ \ln^2 N } \xrightarrow[N \rightarrow \infty]{\mathrm{D}} \chi,
    \]
    where $\chi$ is a r.v. with the density
    \[
        p_{\chi}(t) = \frac{1}{\sqrt{2 \pi \sigma^2 t^3}} \exp\left( -\frac{1}{2 \sigma^2 t} \right) I\{ t \geq 0 \}.
    \]
    \label{theorem:main}
\end{theorem}
\begin{proof}
Consider the random walk $S_n = \sum_{i = 1}^n \xi_i$ with $S_0 = 0$.
This walk is used to represent the branching process $N_n$ by the following recurrent sequence:
\[
    N_{n} = N_k e^{S_{n} - S_{k}} + \sum_{i = k + 1}^n R_{i} e^{S_{n} - S_{i}},\ 0 \leq k < n,
\]
where $R_i = N_{i} - N_{i - 1} e^{\xi_i}$.

Clearly, the increment $S_n - S_i$ is independent of $R_i$ for all $i \leq n$.
Now, we are ready to estimate $R_n$ using the following lemma.
Note that all subsequent statements will be proved under the assumptions of Theorem~\ref{theorem:main}.
\begin{lemma}
    There is a constant $c > 0$ such that
    \[
        \ERB \abs*{ R_n }^{1 + \delta} \leq c \, e^{\zeta_n} \ERB N_{n - 1},\ n \geq 1.
    \]
    \label{lemma:step-bound}
\end{lemma}
\begin{proof}
    Let us use the notation
    \[
        F_n = \sum_{i = 1}^{N_{n - 1}} F_{n, i},\ 
        M_n = \sum_{i = 1}^{N_{n - 1}} M_{n, i}.
    \]
    As seen, $R_n$ can be represented as
    \begin{align*}
        R_n &= L\left( F_n, M_n, \eta_n \right) - g\left( F_n, M_n, \eta_n \right) \\
        &+\ g\left( F_n, M_n, \eta_n \right) - N_{n - 1} g\left( \ERB \varphi_n, \ERB \mu_n, \eta_n \right).
    \end{align*}
    Using~\ref{cond:mating-apprx-apprx}, the first difference can be estimated as
    \begin{align*}
        \ERB \abs*{ L\left( F_n, M_n, \eta_n \right) - g\left( F_n, M_n, \eta_n \right) }^{1 + \delta}
        &\leq \rho^{1 + \delta}(\eta_n) \, \ERB \left( F_n + M_n \right)^{\alpha (1 + \delta)} \\
        &\leq \rho^{1 + \delta}(\eta_n) \, \ERB N_{n - 1} \left( \ERB \varphi_n + \ERB \mu_n \right).
    \end{align*}
    Using~\ref{cond:mating-apprx-scale} and~\ref{cond:mating-apprx-lipsh}, we bound the second difference as
    \begin{align*}
        &\ERB \abs*{ g\left( F_n, M_n, \eta_n \right) - N_{n - 1} g\left( \ERB \varphi_n, \ERB \mu_n, \eta_n \right) }^{1 + \delta} \\
        &\qquad \qquad \qquad \qquad
        \leq \lambda^{1 + \delta}(\eta_n) \, \ERB \left( \abs*{ F_n - N_{n - 1} \ERB \varphi_n } + \abs*{ M_n - N_{n - 1} \ERB \mu_n } \right)^{1 + \delta},
    \end{align*}
    Applying Lemma~\ref{lemma:jen-ineq}, we bound the last expression by
    \[
        \lambda^{1 + \delta}(\eta_n) \, 2^{\delta} \left( \ERB \abs*{ F_n - N_{n - 1} \ERB \varphi_n }^{1 + \delta} + \ERB \abs*{ M_n - N_{n - 1} \ERB \mu_n }^{1 + \delta} \right),
    \]
    which, by Lemma~\ref{lemma:zygmund-marcinkiewicz}, is less than
    \[
        \lambda^{1 + \delta}(\eta_n) \, 2^{\delta} c_{1 + \delta}\, \ERB N_{n - 1} \left( \ERB \abs*{ \varphi_n - \ERB \varphi_n }^{1 + \delta} + \ERB \abs*{ \mu_n - \ERB \mu_n }^{1 + \delta} \right).
    \]

    Thus, applying Lemma~\ref{lemma:jen-ineq}, we bound $\ERB \abs*{ R_n }^{1 + \delta}$ as
    \[
        \ERB \abs*{ R_n }^{1 + \delta} \leq c \, e^{\zeta_n} \ERB N_{n - 1},\ n \geq 1,
    \]
    where $c = 2^{\delta} \left( 1 + 2^{\delta} c_{1 + \delta} \right)$.
\end{proof}

\begin{lemma}
    The inequality 
    \[
        \ERB N_n \leq e^{\xi_n} \ERB N_{n - 1}
    \]
    holds for all $n \geq 1$.
    \label{lemma:bbpre-expectation}
\end{lemma}
\begin{proof}
    Let us fix $\bfeta$ and consider the sequence 
    \[
        b^n_m = \frac{ \ERB[ N_n ][ N_{n - 1} = m ] }{ m },\ m \geq 1.
    \]
    We will show that $b^n_m \leq e^{\xi_n}$ for all $m \geq 1$.
    
    Let us assume the opposite, i.e. there is $m' \geq 1$ such that $b^n_{m'} > e^{\xi_n}$.
    Similarly to the previous lemma, we will use the notation
    \[
        F_n^m = \sum_{i = 1}^{m} F_{n, i}, \ 
        M_n^m = \sum_{i = 1}^{m} M_{n, i}.
    \]
    Then, using~\ref{cond:mating-superadditive}, we get
    \begin{align*}
        \ER[ N_n ][  N_{n - 1} = 2m, \bfeta ]
        &= \ERB L\left( F_n^{2m}, M_n^{2m}, \eta_{n} \right) \\
        &\geq \ERB L\left( F_n^m, M_n^m, \eta_{n} \right) \cdot 2,\ m \geq 1,
    \end{align*}
    and, thus, $b^n_{2^k m'} > e^{\xi_n}$ for all $k \geq 1$, and $\limsup\limits_{m \rightarrow \infty} b^n_m > e^{\xi_n}$.

    At the same time, we will show that $\limsup\limits_{m \rightarrow \infty} b^n_m \leq e^{\xi_n}$, which will lead us to a contradiction.
    Indeed, 
    \[
        b^n_m 
        \leq \frac{ \ERB \abs*{ L\left( F_n^m, M_n^m, \eta_n \right) - g\left( F_n^m, M_n^m, \eta_n \right) } }{ m }
        + \frac{ \ERB g\left( F_n^m, M_n^m, \eta_n \right) }{ m }.
    \]
    Using~\ref{cond:mating-apprx-apprx} and Jensen's inequality, the first term can be bounded  by
    \[
        \rho(\eta_n) \frac{ \ERB^{\alpha} \left( F_n^m + M_n^m \right) }{ m }
        = \rho(\eta_n) \frac{ \ERB^{\alpha} \left( \varphi_n + \mu_n \right) }{ m^{1 - \alpha} }.
    \]
    Now, using~\ref{cond:mating-apprx-scale}, the second term can be replaced by
    \[
        \frac{ \ERB g\left( F_n^m, M_n^m, \eta_n \right) - g\left( m \ERB \varphi_n, m \ERB \mu_n, \eta_n \right) }{ m }
        + g\left( \ERB \varphi_n, \ERB \mu_n, \eta_n \right).
    \]
    Using~\ref{cond:mating-apprx-lipsh} and Lemma~\ref{lemma:zygmund-marcinkiewicz}, we bound the first part by
    \begin{align*}
        &\lambda(\eta_n) \frac{ \ERB \abs*{ F_n^m - m \ERB \varphi_n } + \ERB \abs*{ M_n^m - m \ERB \mu_n } }{ m } \\
        &\qquad \qquad \qquad \qquad
        \leq \frac{ \lambda(\eta_n) c_{1 + \delta}^{\alpha} }{ m^{1 - \alpha} } \left( \ERB^{\alpha} \abs*{ \varphi_n - \ERB \varphi_n }^{1 + \delta} + \ERB^{\alpha} \abs*{ \mu_n - \ERB \mu_n }^{1 + \delta} \right).
    \end{align*}
    Finally, combining the previous bounds, we get
    \[
        b^n_m = g\left(\ERB \varphi_n , \ERB \mu_n, \eta_n \right) + O\left( m^{\alpha - 1} \right) = e^{\xi_n} + o(1), \ m \rightarrow \infty,
    \]
    which leads us to the contradiction.
\end{proof}

Now, we estimate the accumulated error of approximation of $N_n$ by $N e^{S_n}$ using the following lemma.
\begin{lemma}
    There is a constant $c > 0$ such that, for all $n \geq 1$,
    \[
        \ERB \abs{ N_n - N e^{S_n} }^{1 + \delta}
        \leq c\, n^{\delta} N e^{S_n} \sum_{i = 1}^n e^{ \zeta_i - \xi_i + \delta \left( S_n - S_i \right) }.
    \]
    \label{lemma:residual-bound}
\end{lemma}
\begin{proof}
    As seen, the difference $N_n - N e^{S_n}$ can be rewritten as
    \[
        N_n - N e^{S_n} = \sum_{i = 1}^n R_i e^{S_n - S_i}.
    \]
    So, using Lemmas~\ref{lemma:jen-ineq},~\ref{lemma:step-bound} and~\ref{lemma:bbpre-expectation}, we consequently get
    \begin{align*}
        \ERB \abs{ N_{n} - N e^{S_n} }^{1 + \delta}
        &\leq n^{\delta} \sum_{i = 1}^n \ERB[ \abs{ R_i }^{1 + \delta} e^{ \left( 1 + \delta \right) \left( S_n - S_i \right) } ] \\
        &\leq n^{\delta} \sum_{i = 1}^n e^{ \left( 1 + \delta \right) \left( S_n - S_i \right) } \, \ERB \abs{ R_i }^{1 + \delta} \\
        &\leq c\, n^{\delta} \sum_{i = 1}^n e^{ \left( 1 + \delta \right) \left( S_n - S_i \right) } e^{\zeta_i} \, \ERB[ N_{i - 1} ] \\
        &\leq c\, n^{\delta} N e^{S_n} \sum_{i = 1}^n e^{\zeta_i - \xi_i + \delta \left( S_n - S_i \right) }.
    \end{align*}
\end{proof}

Now, we are ready to connect the extinction moment of the process $N_n$ with the hitting time of the corresponding random walk.
Let
\[
    \theta = \min\left\{ n \geq 1 \, \left| \ S_n \leq \ln^{\gamma} N - \ln N \right. \right\},\ \gamma = \frac{ 2 }{ 1 + \beta } < 1.
\]
Firstly, we will write a proof of the well-known fact that $\theta$ is of the order $\ln^2 N$.
Secondly, we will show that $\theta$ is close to $\tau$.

\begin{lemma}
    As $N \rightarrow \infty$,
    \[
        \frac{ \theta }{ \ln^2 N } \xrightarrow[]{\mathrm{D}} \chi,
    \]
    where $\chi$ is a random variable with the density
    \[
        p_{\chi}(t) = \frac{1}{\sqrt{2 \pi \sigma^2 t^3}} \exp\left( -\frac{1}{2 \sigma^2 t} \right) I\{ t > 0 \}.
    \]
    \label{lemma:rw-descent-asymptotics}
\end{lemma}
\begin{proof}
    By the definition of $\theta$, one can state that 
    \[
        \PR[ \frac{ \theta }{ \ln^2 N } \leq x ]
        = \PR[ \frac{ \min_{i \leq x \ln^2 N} S_i }{ \sigma \sqrt{x \ln^2 N} } \leq \frac{ \ln^{\gamma} N - \ln N }{ \sigma \sqrt{x \ln^2 N} } ],
    \]
    where $x > 0$.
    By Donsker's invariance principle, we have
    \[
        \frac{ \min_{i \leq n} S_i }{ \sigma \sqrt{n} } \xrightarrow[n \rightarrow \infty]{\mathrm{D}} \inf_{t \in [0, 1]} W_t,
    \]
    where $W_t$ is a standard Brownian motion.
    At the same time,
    \[
        \frac{ \ln^{\gamma} N - \ln N }{ \sigma \sqrt{x \ln^2 N} } \xrightarrow[N \rightarrow \infty]{} \frac{-1}{\sigma \sqrt{x}}.
    \]
    Thus, by Slutsky's theorem, we obtain
    \[
        \PR[ \frac{\theta}{\ln^2 N}  \leq x ] \xrightarrow[N \rightarrow \infty]{} \PR[ \inf_{t \in [0, 1]} W_t \leq \frac{-1}{\sigma \sqrt{x}} ],
    \]
    for $x > 0$, which is equivalent to the statement of the lemma.
\end{proof}

Now, we will show that $\theta$ approximates $\tau$.
Namely, we will prove that $N_n$ extincts soon after the moment $\theta$.
\begin{lemma}
    For any $\varepsilon > 0$ and $k = \lfloor \varepsilon \ln^2 N \rfloor$,
    we have 
    \[
        \PR[ N_{\theta + k} > 0 ] \rightarrow 0,\ N \rightarrow \infty.
    \]
    \label{lemma:extinction}
\end{lemma}
\begin{proof}
    By the law of total probability, we have
    \begin{align}
        \PRB[ N_{\theta + k} > 0 ]
        &\leq \PRB[ N_{\theta + k} \geq 1 ][ N_{\theta} > 2 N e^{S_{\theta}} ]
        \PRB[ N_{\theta} > 2 N e^{S_{\theta}} ] \nonumber \\
        &+\, \PRB[ N_{\theta + k} \geq 1 ][ N_{\theta} \leq 2 N e^{S_{\theta}} ]
        \PRB[ N_{\theta} \leq 2 N e^{S_{\theta}} ] \label{eq:extinction-total} \\
        &\leq \PRB[ N_{\theta} > 2 N e^{S_{\theta}} ]
        + \PRB[ N_{\theta + k} \geq 1 ][ N_{\theta} = \lfloor 2 N e^{S_{\theta}} \rfloor ]. \nonumber
    \end{align}
    
    Note that $S_{\theta} - S_i \leq 0$ for all $i \leq \theta$, due to the definition of $\theta$. 
    Thus, using Markov's inequality, Lemma~\ref{lemma:residual-bound}, and this fact, consequently, we bound the first term of the right side of~\eqref{eq:extinction-total} in the following way:
    \begin{align}
        \begin{split}
        \PRB[ N_{\theta} > 2 N e^{ S_{\theta} } ]
        &\leq \PRB[ \abs*{ N_{\theta} - N e^{S_{\theta}} }^{1 + \delta} > N^{1 + \delta} e^{ \left( 1 + \delta \right) S_{\theta} } ] \\
        &\leq \frac{ 1 }{ N^{1 + \delta} e^{ \left( 1 + \delta \right) S_{\theta}} } \ERB \abs*{ N_{\theta} - N e^{S_{\theta}} }^{ 1 + \delta } \\
        &\leq \frac{ c\, \theta^{1 + \delta} }{ N^{\delta} e^{ \delta S_{\theta} } } \frac{ \sum_{i = 1}^{\theta} e^{\zeta_i - \xi_i} }{ \theta }.
        \end{split}
        \label{eq:extinction-first}
    \end{align}
    By the definition of $\theta$, 
    \begin{equation}
        e^{ \ln^{\gamma} N + \xi_{\theta} } \leq N e^{S_{\theta}} \leq e^{ \ln^{\gamma} N }.
        \label{eq:theta-definition}
    \end{equation}
    Thus, the right side of~\eqref{eq:extinction-first} can be bounded by
    \[
        c\, \theta^{1 + \delta} \max_{i \leq \theta} \exp\left[ \zeta_i - \xi_i - \delta \left( \ln^{\gamma} N + \xi_{\theta} \right) \right].
    \]

    Further, we will use the notation $\delta' = \delta / 3 > 0$.  
    Let us prove that 
    \begin{equation}
        \PR[ \max_{i \leq \theta} e^{\zeta_i - \xi_i} > e^{ \delta' \ln^{\gamma} N } ] \rightarrow 0,\ N \rightarrow \infty.
        \label{eq:extinction-max}
    \end{equation}
    Indeed, let us fix $x > 0$. 
    By the law of total probability, we obtain
    \begin{align*}
        \PR[ \max_{i \leq \theta} e^{\zeta_i - \xi_i} > e^{ \delta' \ln^{\gamma} N } ]
        &\leq \PR[ \max_{i \leq x \ln^2 N} e^{\zeta_i - \xi_i} > e^{ \delta' \ln^{\gamma} N } ] \\
        &+ \PR \left( \vphantom{ \max_{i \leq x \ln^2 N} e^{\zeta_i - \xi_i} } \theta > x \ln^2 N \right).
    \end{align*}
    Due to Lemma~\ref{lemma:rw-descent-asymptotics}, 
    \begin{equation}
        \PR[ \theta > x \ln^2 N ] \rightarrow 1 - F_{\chi}(x),\ N \rightarrow \infty,
        \label{eq:extinction-conv}
    \end{equation}
    where $F_{\chi}(x)$ is the cumulative distribution function of $\chi$. 
    Since $\zeta_i - \xi_i$, $i \geq 1$ are i.i.d. r.v.,
    \[
        \PR[ \max_{i \leq x \ln^2 N} e^{ \zeta_i - \xi_i } > e^{ \delta' \ln^{\gamma} N } ]
        \leq 1 - \PR[ \vphantom{ \max_{i \leq x \ln^2 N} e^{ \zeta_i - \xi_i } > e^{ \delta' \ln^{\gamma} N } } \zeta_1 - \xi_1 \leq \delta' \ln^{\gamma} N ]^{x \ln^2 N},
    \]
    where the last expression is less than
    \[
        1 - \exp\left( x \ln^2 N\ \PR[ \zeta_1 - \xi_1 > \delta' \ln^{\gamma} N ] \right),
    \]
    by the inequality $\ln(1 - p) \geq -p$, where $0 \leq p \leq 1$.
    Since $\gamma = 2 / (1 + \beta)$, $\ER \abs*{ \zeta_1 - \xi_1 }^{2 / \gamma}$ is finite, and, by Lemma~\ref{lemma:proba-asymptotics}, 
    \[
        \ln^2 N\ \PR[ \zeta_1 - \xi_1 > \delta' \ln^{\gamma} N ] \rightarrow 0,\ N \rightarrow \infty.
    \]
    Thus,
    \[
        \limsup_{N \rightarrow \infty} \PR[ \max_{i \leq \theta} e^{\zeta_i - \xi_i} > e^{ \delta' \ln^{\gamma} N } ]  \leq 1 - F_{\chi}(x),
    \]
    and, taking $x > 0$ large enough, we state~\eqref{eq:extinction-max}.

    Similarly, for any $x > 0$, one can write
    \begin{align*}
        \PR[ e^{ -\delta \xi_{\theta} } > e^{ \delta' \ln^{\gamma} N } ]
        &\leq \PR[ 3 \xi_{\theta} \leq -\ln^{\gamma} N,\, \theta \leq x \ln^2 N ] \\
        &+ \PR[ \theta > x \ln^2 N ].
    \end{align*}
    The first term can be bounded by
    \[
        \sum_{i = 1}^{x \ln^2 N} \PR[ 3 \xi_i \leq -\ln^{\gamma} N ] \leq x \ln^2 N\ \PR[ 3 \xi_1 \leq -\ln^{\gamma} N ].
    \]
    Applying Lemma~\ref{lemma:proba-asymptotics} to the last expression and using~\eqref{eq:extinction-conv}, we get
    \[
        \limsup_{N \rightarrow \infty} \PR[ e^{ -3 \xi_{\theta} } > e^{ \ln^{\gamma} N } ] \leq 1 - F_{\chi}(x),
    \]
    and, taking $x > 0$ large enough, we prove that
    \begin{equation}
        \PR[ e^{ -\delta \xi_{\theta} } > e^{ \delta' \ln^{\gamma} N } ] \rightarrow 0,\ N \rightarrow \infty.
        \label{eq:extinction-shift}
    \end{equation}

    Since $\theta$ is of the order $\ln^2 N$ by Lemma~\ref{lemma:rw-descent-asymptotics}, for any $x > 0$, we obtain
    \begin{equation}
        \PR[ \theta^{1 + \delta} > x \, e^{ \delta' \ln^{\gamma} N } ] \rightarrow 0,\ N \rightarrow \infty,
        \label{eq:extinction-theta}
    \end{equation}
    by Slutsky's lemma.
    Thus, we have
    \begin{align*}
        &\qquad \qquad \quad \quad \ \ \ \, \, \PR[ \theta^{1 + \delta} \max_{i \leq \theta} \exp\left[ \zeta_i - \xi_i - \delta \left( \ln^{\gamma} N + \xi_{\theta} \right) \right] > x ] \\
        &\leq \PR[ \vphantom{ \max_{i \leq \theta} e^{\zeta_i - \xi_i} > e^{ \delta' \ln^{\gamma} N } } e^{ -\delta \xi_{\theta} } > e^{ \delta' \ln^{\gamma} N } ]
        + \PR[ \max_{i \leq \theta} e^{\zeta_i - \xi_i} > e^{ \delta' \ln^{\gamma} N } ]
        + \PR[ \vphantom{ \max_{i \leq \theta} e^{\zeta_i - \xi_i} > e^{ \delta' \ln^{\gamma} N } } \theta^{1 + \delta} > x \, e^{ \delta' \ln^{\gamma} N } ],
    \end{align*}
    which, in combination with~\eqref{eq:extinction-max},~\eqref{eq:extinction-shift}, and~\eqref{eq:extinction-theta}, proves the convergence
    \begin{equation}
        \PRB[ N_{\theta} > 2 N e^{ S_{\theta} } ] \xrightarrow[]{\mathrm{D}} 0,\ N \rightarrow \infty.
        \label{eq:extinction-first-conv}
    \end{equation}
    
    Now, we consider the second term of the right side of~\eqref{eq:extinction-total}.
    Using Markov's inequality, Lemma~\ref{lemma:bbpre-expectation}, and~\eqref{eq:theta-definition}, consequently, we obtain
    \begin{align}
        \PRB[ N_{\theta + k} \geq 1 ][ N_{\theta} = \lfloor 2 N e^{S_{\theta}} \rfloor ]
        &\leq \min_{i \leq k} \ERB[ N_{\theta + i} ][ N_{\theta} = \lfloor 2 N e^{S_{\theta}} \rfloor ] \nonumber \\
        &\leq c\, N e^{S_{\theta}} \exp\left( \min_{i \leq k} \left( S_{\theta + i} - S_{\theta} \right) \right) \label{eq:extinction-second} \\
        &\leq c\, \exp\left( \ln^{\gamma} N + \min_{i \leq k} \left( S_{\theta + i} - S_{\theta} \right) \right). \nonumber
    \end{align}

    Now, we will prove that
    \begin{equation}
        \exp\left( \ln^{\gamma} N + \min_{i \leq k} \left( S_{\theta + i} - S_{\theta} \right) \right) \xrightarrow[]{\mathrm{D}} 0,\ N \rightarrow \infty.
        \label{eq:extinction-upper-conv}
    \end{equation}
    Indeed, by Donsker's invariance principle, for $x < 0$, we get
    \[
        \PR[ \min_{i \leq k} \left( S_{\theta + i} - S_{\theta} \right) > x \sigma \sqrt{k} ] \rightarrow 1 - F(x),\ N \rightarrow \infty,
    \]
    where $F(x)$ is the cumulative distribution function of $\inf_{t \in [0, 1]} W_t$.
    But, for any $x, t < 0$, there is $N'$ such that $x \sigma \sqrt{k} \leq - \ln^{\gamma} N + t$ for all $N \geq N'$.
    Therefore,
    \[
        \limsup_{N \rightarrow \infty} \PR[ \ln^{\gamma} N + \min_{i \leq k} \left( S_{\theta + i} - S_{\theta} \right) > t ] \leq 1 - \sup_{x < 0} F(x) = 0.
    \]

    Thus, we stated~\eqref{eq:extinction-upper-conv}, which, in combination with~\eqref{eq:extinction-second}, proves the convergence
    \begin{equation}
        \PRB[ N_{\theta + k} \geq 1 ][ N_{\theta} = \lfloor 2 N e^{ S_{\theta} } \rfloor ] \xrightarrow[]{\mathrm{D}} 0,\ N \rightarrow \infty.
        \label{eq:extinction-second-conv}
    \end{equation}
    Thus, summing up~\eqref{eq:extinction-first-conv} and~\eqref{eq:extinction-second-conv}, and applying Slutsky's lemma, we get
    \[
        \PRB[ N_{\theta + k} > 0 ] \xrightarrow[]{\mathrm{D}} 0,\ N \rightarrow \infty.
    \]

    Therefore, by the definition of weak convergence, we have
    \[
        \PR[ N_{\theta + k} > 0] = \ER \left( \PRB[ N_{\theta + k} > 0 ] \right) \rightarrow 0,\ N \rightarrow \infty,
    \] 
\end{proof}

Similarly, we prove that $N_n$ is still non-zero at the moment $\theta$.
\begin{lemma}
    The convergence
    \[
        \PR[ N_{\theta} > 0 ] \rightarrow 1,\ N \rightarrow \infty,
    \]
    is valid.
    \label{lemma:existance}
\end{lemma}
\begin{proof}
    Let us consider the probability of the opposite event.
    Using Markov's inequality and Lemma~\ref{lemma:residual-bound}, we get
    \begin{align*}
        \PRB[ N_{\theta} = 0 ]
        &= \PRB[ N_{\theta} - N e^{ S_{\theta} } \leq -N e^{ S_{\theta} } ] \\
        &\leq \PRB[ \abs*{ N_{\theta} - N e^{S_{\theta}} }^{1 + \delta} \geq N^{1 + \delta} e^{ \left( 1 + \delta \right) S_{\theta} } ].
    \end{align*}

    The last expression appears in~\eqref{eq:extinction-first}.
    As seen from the proof of Lemma~\ref{lemma:extinction}, it converges to zero in distribution.
    By the same reasons as before, we have the required convergence.
\end{proof}

Finally, we are ready to complete the proof of Theorem~\ref{theorem:main}.
On the one hand, for a fixed $x \in \mathbb{R}$ we obtain
\begin{align*}
    \PR[ \tau \leq x \ln^2 N]
    &\leq \PR[ \left\{ \theta \leq x \ln^2 N \right\} \cup \left\{ \tau \leq \theta \right\} ] \\
    &\leq \PR[ \theta \leq x \ln^2 N ] + \PR[ N_{\theta} = 0 ].
\end{align*}
But, by Lemma~\ref{lemma:existance}, $\PR[ N_{\theta} = 0 ]$ converges to zero as $N \rightarrow \infty$.

On the other hand, for any $\varepsilon > 0$, we get
\begin{align*}
    \PR[ \tau \leq x \ln^2 N ]
    &\geq \PR[ \left\{ \theta \leq (x - \varepsilon) \ln^2 N \right\} \setminus \left\{ \tau \geq \theta + \varepsilon \ln^2 N \right\} ] \\
    &\geq \PR[ \theta \leq (x - \varepsilon) \ln^2 N ] - \PR[ N_{\theta + k} > 0 ],
\end{align*}
where $k = \lfloor \varepsilon \ln^2 N \rfloor$.
Again, due to Lemma~\ref{lemma:extinction}, $\PR[ N_{\theta + k} > 0 ]$ converges to zero as $N \rightarrow \infty$.

Thus, taking $N \rightarrow \infty$, and then, $\varepsilon \rightarrow 0$, we obtain
\[
    \lim_{ N \rightarrow \infty } \PR[ \frac{\theta}{\ln^2 N} \leq x ]
    \leq \lim_{ N \rightarrow \infty } \PR[ \frac{ \tau }{ \ln^2 N } \leq x ]
    \leq \lim_{ N \rightarrow \infty } \PR[ \frac{\theta}{\ln^2 N} \leq x ],
\]
which proves Theorem~\ref{theorem:main}.
\end{proof}


\end{document}